\documentclass[reqno]{amsart}
\usepackage{amssymb,amscd,verbatim, amsthm,graphics, color,latexsym,amsmath,multicol}
\usepackage{fancybox}
\usepackage{longtable}

\usepackage[all]{xy}

\usepackage[top=2.7cm, bottom=2.7cm, left=2.7cm, right=2.7cm]{geometry}

\newcommand \C[1]{{\mathcal #1}}

\newcommand \wti[1]{{\widetilde {#1}}}

\newcommand \bC{{\mathbb C}}

\newcommand \bR{{\mathbb R}}
\newcommand \bZ{{\mathbb Z}}

\newcommand \bN{{\mathbb N}}

\newcommand\CH{{\C H}}

\newcommand\CO{{\C O}}

\newcommand\al{{\alpha}}

\newtheorem{theorem}{Theorem}[section]

\newtheorem{lemma}[theorem]{Lemma}
\newtheorem{proposition}[theorem]{Proposition}
\newtheorem{definition}[theorem]{Definition}

\newtheorem{example}[theorem]{Example}

\newcommand\Hom{\operatorname{Hom}}

\newcommand\Ext{\operatorname{Ext}}
\newcommand\Ind{\operatorname{Ind}}

\newcommand\tr{\operatorname{tr}}

\newcommand\triv{\mathsf{triv}}
\newcommand\sgn{\mathsf{sgn}}
\newcommand\refl{\mathsf{refl}}

\newcommand\Irr{\mathsf{Irr}}

\newcommand\el{\mathsf{ell}}
\newcommand\res{\mathsf{res}}

\newcommand\ind{\mathsf{ind}}
\newcommand\EP{\mathsf{EP}}
\newcommand\temp{\mathsf{temp}}
\newcommand\ps{\mathsf{ps}}
\newcommand\Inn{\mathsf{Inn}}

\newcommand\ep{\mathsf{ep}}

\newcommand\Id{\operatorname{Id}}

\newcommand\proj{\operatorname{proj}}

\def\<{\langle} 
\def\>{\rangle}

\numberwithin{equation}{section}

\begin{document}

\title[The elliptic nonabelian Fourier transform for unipotent representations ]{On the elliptic nonabelian Fourier transform for unipotent representations of $p$-adic groups}

\author
{Dan Ciubotaru}
        \address[D. Ciubotaru]{Mathematical Institute, University of Oxford, Oxford OX2 6GG, UK}
        \email{dan.ciubotaru@maths.ox.ac.uk}

\author{Eric Opdam}
\address[E. Opdam]{Korteweg-de Vries Institute for Mathematics\\Universiteit van Amsterdam\\Science Park 105-107\\ 1098 XG Amsterdam, The Netherlands}
\email{e.m.opdam@uva.nl} 

\begin{abstract}In this paper, we consider the relation between two nonabelian Fourier transforms. The first  one is defined in terms of the Langlands-Kazhdan-Lusztig parameters for unipotent elliptic representations of a split $p$-adic group and the second is defined in terms of the pseudocoefficients of these representations and Lusztig's nonabelian Fourier transform for characters of finite groups of Lie type. We exemplify this relation in the case of the $p$-adic group of type $G_2$.
\end{abstract}


\dedicatory{To Roger Howe, with admiration}



\maketitle

\setcounter{tocdepth}{1}
\tableofcontents

\section{Introduction}
In this paper,  we consider the relation between two nonabelian Fourier transforms: the first defined in terms of the Langlands-Kazhdan-Lusztig parameters for unipotent elliptic representations of a split $p$-adic group and the second, defined in terms of the pseudocoefficients of these representations and Lusztig's nonabelian Fourier transform for characters of finite groups of Lie type. In this introduction, we give a brief outline of the ideas involved, leaving the precise definitions for the main body of the paper.

\smallskip

Let $G$ be a semisimple $p$-adic group. Lusztig \cite{Lu2} defined the category $\C C_u(G)$ of unipotent $G$-representations. An informal characterization of $\C C_u(G)$ is that it is the smallest full subcategory of smooth $G$-representations such that: (a) it contains the irreducible representations with Iwahori-fixed vectors \cite{Bo}, and (b) it is closed under partition into L-packets. Let $R_u(G)$ be the complexification of the Grothendieck group of admissible representations in $\C C_u(G)$. Let $\CH(G)$ be the Hecke algebra of $G$ with respect to a Haar measure on $G$ and let $\overline\CH(G)$ be the cocenter of $\CH(G)$. 

If one is interested in the study of characters of admissible representations, the basic case is that of elliptic tempered representations (\cite{Ar,BDK}). Let  $\overline R_u(G)_\el$ be the elliptic representation space, a certain  subspace of $R_u(G)$ isomorphic to the quotient of $R_u(G)$ by the span of all properly parabolically induced characters. To every elliptic character $\pi\in \overline R_u(G)_\el$, one attaches a pseudocoefficient $f_\pi \in \overline \CH(G)$ (\cite{Ar,Ka2,SS}). The functions $f_\pi$ play an important role in the character formulas for elliptic representations: for example, if $\pi$ is an irreducible square integrable representation, then $f_\pi(1)$ is the formal degree of $\pi$,  e.g., \cite[Proposition III.4.4]{SS}.  Let $\overline\CH(G)_u^\el$ denote the subspace of $\overline\CH(G)$ spanned by the pseudocoefficients $f_\pi$ for $\pi\in \overline R_u(G)_\el$. Thus, we have a map, in fact, an isomorphism, $\overline R_u(G)_\el\to\overline\CH(G)_u^\el $. 

The two spaces involved have natural inner products. The elliptic representations space $\overline R_u(G)_\el$ carries the Euler-Poincar\' e pairing $\EP$ \cite{Ka2,SS}. With respect to $\EP$, the irreducible square integrable representations are an orthonormal set. On the other hand, the space $\overline\CH(G)_u^\el$ can be endowed with a character pairing coming from the ordinary character pairing on the reductive quotients of maximal parahoric subgroups, which are finite groups of Lie type. As noticed in \cite{MW}, the pseudocoefficients can be chosen so that the resulting map $ \res_{u,\el}:\overline R_u(G)_\el\to\overline\CH(G)_u^\el$ is an isometry with respect to the two pairings. 

Next, one defines the nonabelian Fourier transform $\C F\C T_{u,\el}:\overline\CH(G)^\el_u\to \overline\CH(G)^\el_u$ essentially by truncating to the elliptic spaces Lusztig's nonabelian transform \cite{orange} on each reductive quotient of a maximal parahoric subgroup. 

Now suppose that $G$ is an adjoint simple split group. Let $\Inn~G$ denote the class of forms that are inner to $G$.  The irreducible elliptic tempered modules for the groups in $\Inn~G$ are parameterized in terms of Kazhdan-Lusztig parameters \cite{KL,Lu2,Op,Re,Re2,Wa} in the dual complex group $G^\vee$. Define a dual elliptic nonabelian Fourier transform
\begin{equation}
\C F\C T^{\vee}_{u,\el}: \bigoplus_{G'\in \Inn~G}\overline R_u(G')_\el\longrightarrow \bigoplus_{G'\in \Inn~G}\overline R_u(G')_\el
\end{equation} 
by the requirement that the diagram
\begin{equation}\label{main-diagram}\displaystyle{
\xymatrixcolsep{5pc}\xymatrix{
\bigoplus_{G'\in \Inn~G}\overline R_u(G')_\el \ar[d]_{\res_{u,\el}} \ar[r]^{\C F\C T^{\vee}_{u,\el}} &\bigoplus_{G'\in \Inn~G}\overline R_u(G')_\el \ar[d]^{\res_{u,\el}}\\
\bigoplus_{G'\in \Inn~G}\overline\CH(G')^\el_u \ar[r]_{\C F\C T_{u,\el}} &\bigoplus_{G'\in \Inn~G}\overline\CH(G')^\el_u 
} }
\end{equation}
is commutative. The expectation is that $\C F\C T^{\vee}_{u,\el}$ behaves well with respect to the Kazhdan-Lusztig parameters. For example, if $n\in G^\vee$ is a unipotent element and we write $(\Inn~G)_\el^n=\bigoplus_{G'\in \Inn~G}\overline R_u(G')^n_\el$ for the span of all unipotent representations whose KL parameter has unipotent part conjugate to $n$, then we expect that $\C F\C T^{\vee}_{u,\el}$ is block-diagonal: 
\begin{equation}
\C F\C T^{\vee}_{u,\el}=\bigoplus_n \C F\C T^{\vee,n}_{u,\el}
\end{equation}
 with respect to the decomposition $\bigoplus_n (\Inn~G)_\el^n$. Moreover, we expect that the piece $\C F\C T^{\vee,n}_{u,\el}$ is an elliptic version of the nonabelian Fourier transform recently defined in \cite{L5} in terms of the reductive part of the centralizer $Z_{G^\vee}(n)$. In particular, if $n$ is a distinguished unipotent element, then we believe that $\C F\C T^{\vee,n}_{u,\el}$ will transform the parameters in the same way as the original nonabelian Fourier transform \cite{orange} defined in terms of the group of components of  $Z_{G^\vee}(n)$.

\smallskip

One motivation for these expectations is provided by the work of Moeglin and Waldspurger \cite{MW,Wa} where $G$ is the split form of $SO(2n+1)$. In the present paper, we also verify that these conjectures are precisely true when $G$ has type $G_2$. Further evidence is provided by \cite{CO}, which considered a similar picture for the formal degrees, in other words, the ``evaluation at 1'' of the diagram (\ref{main-diagram}).

In \cite{MW}, the commutative diagram (\ref{main-diagram}) was used in order to verify the stability of L-packets for odd orthogonal groups. The diagram involving the two nonabelian Fourier transforms should also  be related to the recent conjectures of Lusztig regarding unipotent almost characters of semisimple $p$-adic groups \cite{KmL,L4,L5}. In fact, we have arrived at this setup from our attempt to express formal degrees of unipotent discrete series representations in terms of certain invariants, which we called elliptic fake degrees in \cite{CO} and which admit a geometric interpretation. It is our hope that this approach will contribute to a better understanding of the relation between the characters of elliptic unipotent representations,  in the form of the local character expansion of Howe \cite{Ho} and Harish-Chandra \cite{HC}, and the geometry of the affine flag manifold as predicted by the conjectures of Lusztig.

\

\noindent{\bf Acknowledgements.} This research was supported in part by the ERC Advanced Grant no. 268105 and by the EPSRC grant EP/N033922/1. We thank the organizers of the conference ``Representation Theory, Number Theory and Invariant Theory: A conference in honor of Roger Howe on the occasion of his 70th birthday" for the invitation to present the results exposited in this paper. We are grateful to Anne-Marie Aubert for drawing our attention to  the paper \cite{MW} and we thank her and Maarten Solleveld for useful comments. We thank Roman Bezrukavnikov for pointing out a mistake in Table \ref{t:G2-res} in a previous version of this paper. We also thank the referee for helpful suggestions.  

\

\noindent {\bf General notation.} If $\C G$ is an algebraic group, denote by $Z(\C G)$ its center. If $x\in \C G$, denote by $Z_{\C G}(x)$ the centralizer of $x$ in $\C G$ and by $Z_{\C G}(x)^0$ its identity component.  Let $\C S_{\C G}(x)=Z_{\C G}(x)/Z_{\C G}(x)^0$ denote the group of components. The same notation applies with a subset $\C J\subset \C G$ in place of $x$ for the simultaneous centralizer of all elements of $\C J$. Let $N_{\C G}(\C J)$ denote the normalizer of $\C J$ in $\C G$.

If $\C G$ is a reductive group, we call an element $x\in \C G$ {\it elliptic} if $Z_{\C G}(x)$ does not contain a nontrivial split torus.

\section{Generalities} Let ${\mathsf k}$ be a nonarchimedean local field of characteristic $0$ with ring of integers $\CO$, prime ideal $\mathfrak p$ and residue field $\CO/\mathfrak p=\mathbb F_q.$  Let $G$ be the ${\mathsf k}$-points of a connected reductive algebraic group defined over ${\mathsf k}$.  By a $G$-representation $(\pi,V)$ of $G$, we understand a smooth $G$-representation. We denote by $\C C(G)$ the category of smooth $G$-representations, and by $\C C^{\text{adm}}(G)$ the full subcategory of admissible representations. Let $R(G)$ denote the complexification of the Grothendieck group of $\C C^{\text{adm}}(G)$. By a parabolic, Levi subgroup, torus of $G$ we will always mean ${\mathsf k}$-parabolic, ${\mathsf k}$-Levi, ${\mathsf k}$-torus, etc. An element of $G$ is called compact if it is contained in a compact subgroup of $G$.

\subsection{The character distribution}
Fix a Haar measure $\mu$ on $G$. Let $\CH(G)$ denote the Hecke algebra of $G$, i.e., the associative algebra of functions $f:G\to \bC$ that are compactly supported and locally constant, with the product given by the $\mu$-convolution. The Hecke algebra $\CH(G)$ is not unital (unless $G=1$).  As it is well-known, there is an equivalence of categories between $\C C(G)$ and the category of nondegenerate $\CH(G)$-modules: if the $G$-representation is $(\pi,V)$, then the action of $\CH(G)$ is via
\[\pi(f)v=\int_G f(x)\pi(x)v~d\mu(x).
\]
Since $f\in\CH(G)$ is locally constant, there exists a compact open subgroup $K$ such that $f(k_1gk_2)=f(g)$ for all $k_1,k_2\in K$ and then $\pi(f)V\subset V^K$. If $\pi$ is admissible, this implies that every $\pi(f)$ has a well-defined trace. Define the distribution
\begin{equation}
\Theta_\pi(f)=\tr \pi(f),\quad f\in\CH(G).
\end{equation} 
Then $\Theta_\pi$ is zero on the span $[\CH(G),\CH(G)]$ of all commutators of $\CH(G).$ Thus, we may regard $\Theta_\pi$ as a functional on the cocenter $\overline\CH(G)=\CH(G)/[\CH(G),\CH(G)]$. 

Let $G_s$ denote the set of semisimple elements of $G$. For $x\in G_s$, let $D_G(x)$ denote the Harish-Chandra discriminant, i.e., the coefficient of $t^{\text{rk}G}$ in $\det((t+1)\Id-\operatorname{Ad}(x))$. Let 
\[G_{sr}=\{x\in G_s\mid D_G(x)\neq 0\}
\]
denote the set of regular semisimple elements of $G$. For example, when $G=GL(n,{\mathsf k})$, this is the set of all matrices whose eigenvalues are all distinct. Let $G_{rsc}$ denote the set of compact elements in $G_{rs}$.

\begin{theorem}[Harish-Chandra \cite{HC}]Let $\pi$ be an admissible $G$-representation. The character distribution $\Theta_\pi$ is represented on $G$ by a locally integrable function $\Theta_\pi$ which is locally constant on $G_{sr}$:
\[\Theta_\pi(f)=\int_G f(x) \Theta_\pi(x) ~d\mu(x),\quad f\in \CH(G).
\]
\end{theorem}

\subsection{Parabolic induction} Let $P=MN$ be a parabolic subgroup of $G$ and $(\rho,V_\rho)$ a  representation of $M$. The (normalized) parabolically induced representation is the left $G$-regular representation on the space
\[\Ind_P^G(\rho)=\{f:G\to V_\rho\mid f(gmn)=\delta_P(mn)^{-1/2} \rho(m)^{-1} f(g),\ g\in G, m\in M, n\in N\},
\]
where $\delta_P$ is the modulus function of $P$. If $\rho$ is an admissible $M$-representation, the image of  $\Ind_P^G(\rho)$ in $R(G)$ does not depend on the choice of $P\supset M$, and therefore, we may denote by
\[i_M^G: R(M)\to R(G),
\]
the linear map defined by $\Ind_P^G.$  Let $A$ be a maximally split torus of $G$. 

\begin{theorem}[{van Dijk \cite{vD}}] Let $P=MN$ be a parabolic subgroup and let $\rho$ be an admissible $M$-representation. Let $\pi=\Ind_P^G(\rho)$ be the parabolically (normalized) induced representation. Then: 
\[\Theta_\pi(g)=\begin{cases}\sum_{w\in W(A,A_\Gamma)} \frac{|D_{M^w}(g)|^{1/2}}{|D_G(g)|^{1/2}}\cdot \Theta_{\rho^w}(g),&\text{ if }g \in G_{sr}\cap M,\\ 0, &\text{if $g$ is not conjugate to an element of $G_{sr}\cap M$},
\end{cases}
\]
where $W(A,A_\Gamma)$ is defined in \cite[page 237]{vD} and $M^w=wM w^{-1},$ $\rho^w(m)=\rho(w^{-1}mw).$
\end{theorem}
This allows one to compute the character of a parabolically induced representation from the character of the inducing representation of the Levi subgroup. For example, when $G$ is split and $P=B=AN$, if $\chi$ is any smooth character of $A$ and $\pi_\chi=\Ind_B^G(\chi)$ is the corresponding minimal principal series, then
\begin{equation}
\Theta_{\pi_\chi}(g)=\begin{cases}|D_G(a)|^{-1/2}\sum_{w\in W_0} \chi^w(a),&\text{ if }g\text{ is conjugate to an element of }A\cap G_{sr},\\
0,&\text{ otherwise.}\end{cases}
\end{equation}
Here $W_0=N_G(A)/A$ is the finite Weyl group.

\subsection{Jacquet functors} If $(\pi,V)$ is a $G$-representation and $P=MN$ a parabolic subgroup, define $V(N)=\{\pi(n)v-v\mid n\in N,~v\in V\}$ and $V_N=V/V(N)$. Then $V_N$ is an $M$-representation. The Jacquet functor
\[ r_P: \ V\mapsto V_N
\]
maps admissible $G$-representations to admissible $M$-representations by the well-known result of Jacquet.
Fix a minimal parabolic subgroup $P_0=M_0N_0$ of $G$. By a standard parabolic subgroup, we mean a parabolic subgroup $P$ of $G$ such that $P\supset P_0$. A standard Levi subgroup is a Levi subgroup of a standard parabolic. Let $\C L$ denote the set of standard Levi subgroups of $G$.

 For every $M\in\C L$, denote the Jacquet functor at the level of the Grothendieck groups by
 \[ r^M_G: R(G)\to R(M),\quad V\mapsto r_P(V).\]
For the properties of $i_M^G$ and $r_G^M$ and the combinatorial relations between them, see \cite{BDK,Da}.

\subsection{Unramified characters} Let $G^0$ denote the subgroup of $G$ generated by all of the compact open subgroups of $G$. For example, when $G=GL(n,{\mathsf k})$, $G^0=\{g\in G\mid \det g\in \CO^\times\}$. The subgroup $G^0$ is normal in $G$ and the quotient $G/G^0$ is a free abelian group of finite rank. A character $\chi: G\to\bC^\times$ is called {\it unramified} if $\chi |_{G^0}=1$. Let
\begin{equation}
\Psi(G)=\Hom_\bZ(G/G^0,\bC^\times)
\end{equation}
denote the torus of unramified characters of $G$. The same definitions and notations apply to a Levi subgroup $M$. If $M$ is a Levi subgroup, denote  $d(M)=\dim\Psi(M).$ Notice that $d(G)=0$ if and only if the center of $G$ is anisotropic.

The torus $\Psi(G)$ acts on $R(G)$ by ``unramified twists": $\pi\mapsto \pi\psi,$ $\pi\in R(G)$, $\psi\in \Psi(G).$

\subsection{Orbital integrals} Let $\omega\subset G_{sr}$ be a conjugacy class. Let $x$ be a representative of $\omega$ and let $\Gamma=Z_G(x)$, a Cartan subgroup of $G$. For every $f\in \CH(G)$, define the orbital integral 
\begin{equation}
\Phi(x,f)=\int_{G/\Gamma}f(gxg^{-1}) ~dg/d\gamma,
\end{equation}
where $d\gamma$ is a Haar measure on $\Gamma$. The normalization of Haar measures is important, but we will only be interested in the vanishing of such orbital integrals, and we ignore the normalization issue here. Denote also $\Phi(\omega,f)$ in place of $\Phi(x,f).$

It is easy to see that if $f\in [\CH(G),\CH(G)]$, then $\Phi(\omega,f)=0$ for all $\omega\subset G_{sr}$. Hence, it makes sense to consider $\Phi(\omega,f)$ for $f\in\overline\CH(G)$. 

\section{Elliptic representations}
Define the space of induced representations:
\[
R(G)_\ind=\sum_{M\in\C L,~M\neq G} i_M^G(R(M)),
\]
and let the space of (virtual) elliptic representations to be the quotient
\begin{equation}
\overline R(G)_\el=R(G)/R(G)_\ind.
\end{equation}
The space $\overline R(G)_\el$ can be identified naturally with a subspace of $R(G)$ as in \cite[\S 5.5]{BDK}.  Let 
\begin{equation}
A=\frac 1p A_{d(M_0)}\circ A_{d(M_0)-1}\circ\dots\circ A_{d(G)+1}: R(G)\to R(G),
\end{equation}
(when $G$ is semisimple, $d(G)=0$) where \[A_d=\prod_{M\in \C L,d(M)=d} (i_M^G\circ r_G^M-p_M\Id),\]
 $p_M=|N_{W_G}(M)/W_M|$, $p=\prod_{M\in \C L, M\neq G} (-p_M)$, $W_M=N_M(M_0)/M_0.$ Using the combinatorics of the induction and restriction maps \cite[\S5.4]{BDK} (see also \cite[Proposition 2.5 (i)]{Da}), one sees that $\ker A=R(G)_\ind$ and therefore
 \begin{equation}
 \overline R(G)_\el\cong A(R(G)).
 \end{equation}
The action of the torus $\Psi(G)$ on $R(G)$ preserves $R(G)_\ind$ and therefore $\Psi(G)$ acts on $\overline R(G)_\el$. For every compact open subgroup $K$ of $G$, let $R(G)_K$ denote the span of all irreducible admissible $G$-representations $(\pi,V)$ such that $V^K\neq 0.$ Then $\Psi(G)$ acts on $R(G)_K$, since unramified characters are trivial of $K$. Let $\overline R(G)_{\el,K}$ denote the image of $R(G)_K$ in $\overline R(G)_\el.$

\begin{theorem}[{\cite[\S3.1]{BDK}}] For every compact open subgroup $K$, $\overline R(G)_{\el,K}$ is a finitely generated $\Psi(G)$-module. In particular, if $G$ is semisimple, then $\overline R(G)_{\el,K}$ is finite dimensional.
\end{theorem}

Notice that, as a consequence of the Langlands classification, every class in $\overline R(G)_\el$ is represented by an essentially tempered module.

\subsection{Projective modules} Let $K_0(G)$ denote the complexification of the Grothendieck group of finitely generated projective $\CH(G)$-modules. A typical element in $K_0(G)$ is the compactly induced module:
\[\ind_K^G(\rho)=\CH(G)\otimes_{\CH(K)}V_\rho,\]
where $K$ is a compact open subgroup and $(\rho,V_\rho)$ is a finite dimensional $K$-representation. In fact, based on the results of Schneider-Stuhler \cite{SS}, Blanc-Brylinski \cite{BB}, and Bernstein, one sees (\cite[Corollary 4.22]{Da}) that $K_0(G)$ is spanned by the modules $\ind_K^G(\rho)$ as $K$ varies over a base of compact open subgroups of $G$ and $\rho$ varies over the irreducible smooth $K$-representations.

The algebra $\CH(G)$ is not unital, but it is a direct limit of unital associative algebras, $\CH(G)=\varinjlim_{K}\CH(G,K)$, where $K$ varies over an appropriate base of compact open subgroups (see \cite{Ka} for example), and this allows one to introduce for $\CH(G)$-modules standard techniques from homological algebra of asssociative unital algebras. 

\begin{definition}
If $P$ is a finitely-generated projective $\CH(G)$-module, there exists $n\in \bN$ and $e_p$ an idempotent in $M_n(\CH(G))$ such that $P=\CH(G)^n e_P$.  Define the \emph{Hattori-Stallings trace map} $\tau_{HS}: K_0(G)\to \overline \CH(G)$ to be
\begin{equation}
\tau_{HS}(P)=\tr e_P \text{ mod } [\CH(G),\CH(G)].
\end{equation}
This definition is independent of the choice of idempotent $e_P$. 

Define the \emph{Euler-Poincar\'e map} $\ep:R(G)\to K_0(G)$ by
\begin{equation}
\pi\mapsto \sum_{i=1}^k (-1)^i[P_i],
\end{equation}
where $0\to P_k\to \dots\to P_0\to \pi\otimes \bC[G/G^0]\to 0$ is a resolution by finitely generated projective modules. It is known that $\CH(G)$ has finite cohomological dimension \cite[Theorem 29]{Be}, \cite[Corollary II.3.3]{SS}.
\end{definition}

Every $f\in \CH(K)$, $K$ a compact open subgroup, can be extended by zero to a function in $\CH(G)$ which we will denote $\wti f$. To understand the map $\tau_{HS}$ better, notice that if $P=\ind_K^G(\rho)\in K_0(G)$, where $\rho$ is irreducible, then $P=\CH(G)  \wti e_\rho$, where $e_\rho\in \CH(K)$ is the idempotent corresponding to $\rho$. This means that 
\[\tau_{HS}(\ind_K^G(\rho))=\wti e_\rho \text{ mod } [\CH(G),\CH(G)].
\]
Consider the composition $\tau_{HS}\circ \ep: R(G)\to \overline \CH(G)$. By results of \cite{Ka,SS}, the space of induced representations $R(G)_\ind$ lies in the kernel of this composition, hence we have a map
\begin{equation}
\tau_{HS}\circ \ep: \overline R(G)_\el\longrightarrow \overline\CH(G).
\end{equation}
To describe the image of this map, we need to define the duals in $\overline\CH(G)$ of the induction and Jacquet functors. Let $R(G)^*_{\text{good}}$ denote the space of good forms defined in \cite{BDK}. By the Density Theorem  of Kazhdan \cite[Theorem 0]{Ka2} (see also \cite[Theorem B]{Ka}) and the Trace Paley-Wiener Theorem of Bernstein, Deligne, and Kazhdan \cite[Theorem 1.2]{BDK},  the trace map
\begin{equation}
\tr: \overline\CH(G)\to R(G)^*_{\text{good}},\ \tr(f)(\pi)=\Theta_\pi(f)
\end{equation}
is an isomorphism. We can use this isomorphism to define the dual maps using the trace pairing. Namely, for every standard Levi subgroup $M$, let 
\begin{equation}
\begin{aligned}
\bar i_M^G: \overline \CH(M)\to \overline \CH(G),\ \Theta_\pi(\bar i_M^G(f))=\Theta_{r^M_G(\pi)}(f),\quad \pi\in R(G),\ f\in\CH(M),\\
\bar r^M_G: \overline \CH(G)\to \overline \CH(M),\ \Theta_\pi(\bar r^M_G(f))=\Theta_{i_M^G(\pi)}(f),\quad \pi\in R(M),\ f\in\CH(G).
\end{aligned}
\end{equation}

The following results are collected in \cite{Da} and they are based on \cite{BDK,Ka2,SS,BB}.

\begin{theorem}\label{t:Dat}
\begin{enumerate}
\item {\cite[Theorem 1.6 and Corollary 4.20]{Da}} The map $\tau_{HS}: K_0(G)\to \overline \CH(G)$ is injective. Its image is the subspace
\begin{equation}
\overline\CH_c(G)=\{f\in\overline\CH(G)\mid \Phi(\omega,f)=0 \text{ for all  conjugacy classes }\omega\subset G_{sr}\setminus G_{src}\}.
\end{equation}
\item {\cite[Theorem 3.3 and Theorem 3.4]{Da}} The map $\tau_{HS}\circ \ep$ induces a linear isomorphism of $\overline R(G)_\el$ onto the \emph{elliptic cocenter}
\begin{equation}
\overline \CH(G)^\el=\bigcap_{M\in \C L, d(M)>d(G)} \ker \bar r^M_G.
\end{equation}
Moreover, $\overline \CH(G)^\el=\{f\in\overline\CH_c(G)\mid \Phi(\omega,f)=0 \text{ for all  nonelliptic conjugacy classes }\omega\subset G_{sr}\}.$

\end{enumerate}
\end{theorem}

In light of Theorem \ref{t:Dat}, we make the following definition.

\begin{definition}
Define the \emph{pseudo-coefficient map} $\ps: \overline R(G)_\el\to \overline \CH^\el$ to be the isomorphism from Theorem \ref{t:Dat}(2). For every class $[\pi]\in \overline R(G)_\el$, let $f_\pi=\ps(\pi)$ and call it the \emph{Euler-Poincar\' e function} of $\pi$. 
\end{definition}
It is easy to see from the definitions that
\begin{equation}
\Theta_{\pi'}(f_{\pi})=\EP(\pi,\pi'), \text{ where }\EP(\pi,\pi')=\sum_{i\ge 0}(-1)^i\dim \Ext_{\CH(G)}^i(\pi,\pi').
\end{equation}
The sum in the right hand side is finite since $\CH(G)$ has finite cohomological dimension, as mentioned before, and $\dim \Ext_{\CH(G)}^i(\pi,\pi')<\infty$ for all smooth admissible representations $\pi,\pi'$.

\subsection{Euler-Poincar\' e functions} We recall an explicit description of $f_\pi$ from \cite{SS}. For the applications in this paper, it is sufficient to assume that $G$ is semisimple and split, which, for brevity, we assume now. Let $A$ be a maximal split torus and $B=AN$ a Borel subgroup of $G$. Let  $R$ be the set of roots of $A$ in $G$ and let $\Pi$ be the set of simple roots of $A$ in $G$ with respect to $B$. Let 
\begin{equation}
W_0=N_G(A)/A \text{ and } W=N_G(A)/A(\CO)
\end{equation}
be the finite Weyl group and the Iwahori-Weyl group, respectively. Let $W\longrightarrow W_0$ be the natural projection whose kernel is the lattice $X=A/A(\CO)\cong \Hom({\mathsf k}^\times, A)$. Let $Y=\Hom(A,{\mathsf k}^\times)$ be the character lattice. For every $x\in X,\ y\in Y$, the composition $y\circ x: {\mathsf k}^\times\to {\mathsf k}^\times$ is an algebraic homomorphism, hence it is given by $z\mapsto z^n$ for some $n\in\bZ$. Denote $\langle x,y\rangle:=n$; this defines a perfect pairing between $X$ and $Y$. Let $R^\vee\subset X$ be the set of coroots in duality $\al\leftrightarrow \al^\vee$, $\langle\al,\al^\vee\rangle=2$ with the roots $R$.

Identify $W=W_0\ltimes X$.  A typical element of $W$ is $w t_x$, where $w\in W_0$ and $x\in X$ and the product is $(w_1 t_{x_1})\cdot (w_2 t_{x_2})=w_1w_2 t_{w_2^{-1}(x_1)+x_2}.$  The group $W$ acts naturally on $Y\times \bZ$, where $Y=\Hom(A,{\mathsf k}^\times)$, via:
\[wt_x: (y,k)\mapsto (w(y),k-\langle x,y\rangle).
\]
Define the positive and negative affine roots:
\begin{equation*}
\begin{aligned}
R_{a}^+&=\{(\al,k)\mid \al\in R,~k>0\}\cup \{(\al,0)\mid \al\in R^+\}\\
R_a^-&=\{(\al,k)\mid \al\in R,~k<0\}\cup \{(\al,0)\mid \al\in R^-\},
\end{aligned}
\end{equation*}
and set $R_a=R_a^+\cup R_a^-=R\times \bZ.$ This is a $W$-stable subset of $Y\times\bZ$. The simple affine roots are 
\[\Pi_a=\{(\al,0)\mid \al\in\Pi\}\cup\{(\al,1)\mid \al\in \Pi_m\}\subset R_a^+,\]
where $\Pi_m=\{\beta\in R\mid \beta\text{ minimal with respect to }\le\}$ where $\beta_1\le\beta_2$ if, by definition, $\beta_2-\beta_1$ is a nonnegative integral combination of roots in $\Pi.$

The group $W$ acts on the space $E=\bR\otimes_\bZ X$ via 
\[ wt_x: e\mapsto w(e+x). 
\]
Let $W^a=W_0\ltimes \bZ R^\vee$ be the affine Weyl group. It is generated by the reflections corresponding to the simple affine roots. 

Assume for simplicity that $G$ is simple. Then we may regard $E$ as a simplicial complex with hyperplanes corresponding to the affine reflections, then $W^a$ acts transitively on the open facets of $E$. If $c\subset E$ is a facet, denote by $W_c$ the stabilizer of $c$ in $W$. This is a finite group. 

There exists a unique open facet $c_0$ (the fundamental alcove) which contains $0$ in its closure and such that all simple roots $\al\in\Pi$ take positive values on $c_0$. The reflections in the hyperplanes that border $c_0$ are the simple affine reflections corresponding to the roots in $\Pi_a$.

 Let $\Omega=W_{c_0}$. Then \[W=\Omega\ltimes W^a.\]
For the definition of the building $\C B_G$ of $G$, see \cite{SS,Ti} for example. This is a simplicial complex containing $E$ as a subcomplex and such that $G$ acts on $\C B_G$ via simplicial maps. Every facet of $\C B_G$ can be translated into $E$ by the action of $G$. Define
\begin{equation}
P_c^+=\text{ stabilizer in }G\text{ of the facet }c\in\C B_G.
\end{equation}
This group sits into a short exact sequence
\[1\longrightarrow U_c\longrightarrow P_c^+\longrightarrow M_c^+\longrightarrow 1,\]
where $U_c$ is the prounipotent radical and $M_c^+$ is the $\mathbb F_q$-points of a possibly disconnected reductive group over $k$. Let $M_c=M_c^{+,0}$ be the $\mathbb F_q$-points of the identity component of $M_c^+$. Define $P_c$ to be the inverse image in $P_c^+$ of $M_c$. This is called a parahoric subgroup of $G$; it is a compact open subgroup.

Let $I=P_{c_0}$, an Iwahori subgroup. Suppose $c$ is a facet of $\C B_G$ such that $c\subset \overline c_0$. Then $P_c\supset I$ is a standard parahoric subgroup. We have $P_c^+/P_c\cong \Omega_c$, where $\Omega_c$ is the stabilizer of $c$ in $\Omega.$ 
The assignment
\begin{equation}
J\subsetneq \Pi_a \longrightarrow P_J:=P_{c_J}
\end{equation}
gives a one-to-one correspondence between $\Omega$-orbits of proper subsets in $\Pi_a$ and $G$-conjugacy classes of parahoric subgroups. Here $c_J$ is a the facet contained in the closure of $c_0$ and defined by the hyperplanes corresponding to the roots in $J$:
\[P_J=\langle I, I{s_\al} I: \al\in J\rangle.\]
At one extreme, $P_{\emptyset}=I$ and at the other, if  $J$ is maximal, then $P_J$ is the stabilizer of a vertex of $\overline c_0$ and it is a maximal parahoric subgroup of $G$.

Using an explicit resolution for admissible representations $\pi$ in terms of projective modules that are compactly induced from parahoric subgroups, Schneider and Stuhler obtain the following explicit description of the functions $f_\pi$.

\begin{theorem}[{\cite[Theorem III.4.20]{SS}}]
Let $(\pi,V)$ be an admissible representation of $G$. Set 
\begin{equation}\label{e:f-pi}
F_\pi=\sum_{J\subsetneq \Pi_a/\Omega} (-1)^{|J|} \frac{\wti\epsilon_J \wti\chi_J}{\mu(P_J^+)},
\end{equation}
where $\chi_J$ is the character of the $P_J^+$-module $V^{U_J}$ (this is zero or finite dimensional),  $\epsilon_J: P_J^+\to\{\pm 1\}$ is the orientation character, and $\wti\epsilon_J,\wti\chi_J$ denote the extension by zero. Then $f_\pi=F_\pi$ mod $[\CH(G),\CH(G)]$.
\end{theorem}

The relevance of the functions $F_\pi$ is that they give an explicit incarnation of the {\it pseudo-coefficients} of square integrable representations $\pi\in \overline R(G)_\el$. 

\begin{theorem}[{\cite[Theorem III.4.6]{SS}, \cite[Corollary p. 29]{Ka2}}]
Suppose that $\pi$ is an irreducible square integrable representation of $G$ and that $\pi'$ is an irreducible tempered $G$-representation. Then
\[\Theta_{\pi'}(f_\pi)=\begin{cases}1,& \pi'\cong\pi,\\
0,&\pi'\not\cong \pi.
\end{cases}
\]
The Euler-Poincar\' e pairing $\EP$ is nondegenerate on $\overline R(G)_\el$ and the set of irreducible square integrable representations form an orthonormal set with respect to $\EP$. 
\end{theorem}

In fact, more is true: with the notation as in the previous theorem, $\Ext_{\CH(G)}^i(\pi,\pi')=0$ if $i>0$ \cite[Theorem 41]{Mey}. This result has been extended in \cite[Theorem 3]{OS2} to a complete determination of $\Ext^i_{\CH(G)}(\pi,\pi')$ for all irreducible tempered modules, in terms of the Knapp-Stein
theory of standard intertwiners and analytic R-groups.

\section{Unipotent representations}\label{s:unipotent}

\subsection{Weyl groups} Let $\C W$ be a finite Weyl group acting on the $n$-dimensional reflection representation $V$. Let $R(\C W)$ be the complexification of the Grothendieck group of finite dimensional $\C W$-representations. It can be naturally identified with the space of complex class functions on $\C W$. Let $\langle~,~\rangle_{\C W}$ be the character pairing of $\C W$ on $R(\C W)$. Let $\Irr~\C W$ be the set of irreducible characters of $\C W$.

Define $\bigwedge^{-} V=\sum_{i=0}^n (-1)^i\bigwedge^i V$, viewed as a virtual $\C W$-representation. Its character is $(\tr \bigwedge^{-} V) (w)= \det_V(1-w)$. Call an element $w\in \C W$ {\it elliptic} if $\det_V(1-w)\neq 0$. It is easy to see that $w$ is elliptic if and only if $w\notin \C W_J$ for every proper parabolic subgroup $\C W_J\subset \C W.$ Define the hermitian pairing
\begin{equation}
\langle~,~\rangle_{\C W}^\el: R(\C W)\times R(\C W)\to \bC,\quad \langle\sigma,\sigma'\rangle_{\C W}^\el=\langle\sigma,\sigma'\otimes {\bigwedge}^{-}V\rangle_{\C W},
\end{equation}
and call it the elliptic pairing of $\C W$. The radical of $\langle~,~\rangle_{\C W}^\el$ is $R(\C W)_\ind$ \cite[Proposition 2.2.2]{Re}, the span of induced representations from proper parabolic subgroups. Let $\overline R(\C W)_\el=R(\C W)/R(\C W)_\ind$ be the space of elliptic representation. It may be naturally identified with the space of class functions on $W$ supported on the elliptic conjugacy classes of $W$.

\subsection{Finite groups}\label{s:finite-groups} Let $\C G$ be a connected semisimple algebraic $\overline {\mathbb F}_q$-group defined over ${\mathbb F}_q$ with Frobenius morphism $F: \C G\to \C G$. Let $\C G^F$ be the fixed points of $F$, a finite group of Lie type. We assume that $\C G$ is adjoint and $\C G^F$ is split. As before, let $R(\C G^F)$ denote the complex class functions on $\C G^F$ with $\langle~,~\rangle_{\C G^F}$ the character pairing. Let $\Irr ~\C G^F$ denote the set of characters of irreducible representations of $\C G^F$.

To each pair $(T,\theta)$, $T$ an $F$-stable maximal torus of $\C G$ and $\theta: T^F\to \bC^\times$ a character, Deligne and Lusztig \cite{DL} associated a {\it generalized character} $R_{T,\theta}\in R(\C G^F)$. 

\begin{definition} A character $\rho\in \Irr ~\C G^F$ is called \emph{unipotent} if $\langle\rho, R_{T,1}\rangle_{\C G^F}\neq 0$ for some $F$-stable maximal torus $T$. Let $\Irr_u\C G^F$ denote the set of irreducible unipotent characters and $R_u(\C G^F)\subset R(\C G^F)$ their span. 

A character $\rho\in \Irr_u\C G^F$ is called \emph{cuspidal} if $\langle\rho, R_{T,\theta}\rangle_{\C G^F}=0$ for all maximal tori $T$ contained in a proper $F$-Levi subgroup.
\end{definition}

Let $T$ be an $F$-stable maximal torus of $\C G$ and suppose that $T$ is contained in an $F$-stable parabolic subgroup $P=M U_P$ with $M$ an $F$-stable Levi subgroup. By \cite[Proposition 2.6]{Lu-cbms}, if $\theta$ is a character of $T^F$, then
\begin{equation}\label{RT-ind}
R_{T,\theta}=\Ind_{P^F}^{\C G^F}(R_{T,\theta}^M),
\end{equation}
where $R_{T,\theta}^M$ is the generalized character of $M^F$, regarded as a virtual representation of $P^F$, trivial on $U_P^F$.  

By \cite[Corollary 3.2]{Lu-cbms}, if $\rho\in \Irr_u\C G^F$ and $\theta'$ is a nontrivial character of $T'^F$, where 
$T'$ is an $F$-stable maximal torus of $\C G$, then
\begin{equation}\label{non-mix}
\langle \rho, R_{T',\theta'}\rangle_{\C G^F}=0.
\end{equation}
Combining (\ref{RT-ind}) and (\ref{non-mix}), it follows that the functors of parabolic induction and restriction take unipotent representations to unipotent representations.

\medskip

Fix a maximal $F$-stable torus $T_0$ contained in an $F$-stable Borel subgroup $B_0$, and let $\C W=N_{\C G^F}(T_0)/T_0$ be the finite Weyl group. As it is well known, the irreducible representations occurring in the minimal principal series $\Ind_{B_0^F}^{\C G^F}(1)$ (which equals $R_{T_0,1}$ by (\ref{RT-ind})) are in one-to-one correspondence with the irreducible representations of $\C W$. If $\mu\in \Irr~\C W$, denote by $\rho_\mu\in \Irr_u \C G^{F}$ the corresponding constituent of the minimal principal series.

Lusztig \cite{orange} partitioned $\Irr~\C W$ into families $\C F$ and attached to each family $\C F$ a finite group $\Gamma_{\C F}$. He defined the sets
\begin{equation}
M(\Gamma_{\C F})=\Gamma_{\C F}\text{-orbits on } \{(x,\sigma)\mid x\in \Gamma_{\C F},\ \sigma\in \Irr~Z_{\Gamma_{\C F}}(x)\}.
\end{equation}

\begin{definition}
The \emph{nonabelian Fourier transform} is the pairing $\{~,~\}: M(\Gamma_{\C F})\times M(\Gamma_{\C F})\to \bC$ given by
\begin{equation}
\{(x,\sigma),(y,\tau)\}=\frac 1{|Z_{\Gamma_{\C F}}(x)||Z_{\Gamma_{\C F}}(y)|}\sum_{g\in\Gamma_{\C F},~xgyg^{-1}=gyg^{-1}x}\sigma(gyg^{-1})~\overline{\tau(g^{-1}xg)}.
\end{equation}
Set $\C X(\C W)=\bigsqcup_{\C F\subset \Irr~\C W}M(\Gamma_{\C F})$ and extend $\{~,~\}$ to $\C X(\C W).$
\end{definition}

On the other hand, to every $\mu\in \Irr~\C W$, Lusztig \cite{orange} associates an \emph{almost character} $R_\mu$ as follows. For every $w\in \C W$, choose an $F$-stable representative $\dot w\in \C G$ and $g\in \C G$ such that $g^{-1}F(g)=\dot w$. Define $T_w=gT_0g^{-1}$, an $F$-stable maximal torus. Define
\begin{equation}
R_\mu=\frac 1{|W|} \sum_{w\in W} \mu(w) R_{T_w,1}.
\end{equation}

\begin{theorem}[{\cite[Theorem 4.23]{orange}}]\label{t:finite-lusztig}
There is a bijection $\C X(\C W)\longrightarrow \Irr_u \C G^F$, $(x,\sigma)\mapsto \rho_{(x,\sigma)}$ such that, for all $(x,\sigma)\in \C X(\C W)$ and $\mu\in\Irr~\C W$:
\[\langle\rho_{(x,\sigma)},R_\mu\rangle_{\C G^F}=\Delta(x,\sigma) \{(x,\sigma),(y,\tau)\},
\]
where $\rho_{(y,\tau)}=\rho_\mu$ and $\Delta(x,\sigma)\in \{\pm 1\}$ is defined in \cite[\S 4.21]{orange}.
\end{theorem}

If $W$ is an irreducible Weyl group, then $\Delta\equiv 1$ except for the families that contain the representations $512_a'$ in $E_7$ or $4096_z$, $4096_x'$ in $E_8$. In light of this theorem, Lusztig \cite[(4.24.1)]{orange} defines the almost character $R_{(y,\tau)}$ for each $(y,\tau)\in \C X(\C W)$:
\begin{equation}
R_{(y,\tau)}=\sum_{\rho_{(x,\sigma)}\in \Irr_u \C G^F}\{(x,\sigma),(y,\tau)\}\Delta(x,\sigma)\rho_{(x,\sigma)}.
\end{equation}
When $\rho_{(y,\tau)}=\rho_\mu$ then $R_{(y,\tau)}=R_\mu$. The set $\{R_{(y,\tau)}:(y,\tau)\in \C X(\C W)\}$ is another orthonormal basis of $R_u(\C G^F)$ \cite[Corollary 4.25]{orange} and the change of basis matrix between this and the basis of irreducible characters is $\{~,~\}$. 

By transporting $\{~,~\}$ via the bijection $\C X(\C W)\longrightarrow \Irr_u \C G^F$ and extending it sesquilinearly, we can define the nonabelian Fourier transform, denoted also by $\{~,~\}$, on $R_u(\C G^F)$. We may also think of it as a linear map:
\begin{equation}
\C F\C T: R_u(\C G^F)\to R_u(\C G^F),\quad \rho\mapsto \sum_{\rho'\in \Irr_u \C G^F}\{\rho,\rho'\} \rho'.
\end{equation}
In \cite{char-sheaves}, Lusztig introduced another set of class functions on $\C G^F$, the characteristic functions $\chi_A$  of $F$-stable character sheaves $A$ on $\C G$ (certain simple perverse sheaves on $G$) and proved \cite[V, Theorem 25.2]{char-sheaves}, under certain restrictions on the characteristic $p$, that they form an orthonormal basis of the space of class functions on $\C G^F$. In particular, restricting to the space $R_u(\C G^F)$,  the set of characteristic functions of unipotent character sheaves gives an orthonormal basis. By further work of Lusztig and Shoji \cite{Sh}, the $F$-stable character sheaves on $\C G$ are parameterized by the same set $\C X(\C W)$ in such a way that if $A_{(y,\tau)}$ is the character sheaf corresponding to $(y,\tau)\in\C X(\C W)$, then 
\begin{equation}\label{almost-chars}
\chi_{A_{(y,\tau)}}=\zeta_{(y,\tau)} \cdot R_{(x,\sigma)},
\end{equation}
for some root of unity $\zeta_{(y,\tau)}$. These roots of unity are explicitly known in most cases, see \cite{L7,Sh,Sh-odd,Sh-even,DLM}. The behavior of the $\chi_A$'s with respect to induction and restriction is studied in \cite[III, Section 15]{char-sheaves}, particularly Proposition 15.7 in {\it loc. cit.}. These results in conjunction with (\ref{almost-chars}) allows one to see that, in particular, the nonabelian Fourier transform $\{~,~\}$ preserves the subspace of $R_u(\C G^F)$ spanned by the proper parabolically induced representations. We will need this fact in subsection \ref{s:finite-elliptic}.

\subsection{Finite $G_2$}
Consider the case when $G$ is the group of type $G_2$ defined over $\mathbb F_q$. Then the finite Weyl group $W(G_2)$ has $6$ irreducible representations, four $1$-dimensional and two $2$-dimensional. In the notation of Carter \cite{Ca}, these are labelled $\phi_{(1,0)}$, $\phi_{(1,3)}'$, $\phi_{(1,3)}''$, $\phi_{(1,12)}$, $\phi_{(2,1)}$, and $\phi_{(2,2)}$. Let $s_1$ denote the long reflection and $s_2$ the short reflection of $W(G_2)$. The character table is Table \ref{t:W(G2)}.

\begin{table}[h]\label{t:W(G2)}
\caption{Character table of $W(G_2)$}
\begin{tabular}{|c|c|c|c|c|c|c|}
\hline
&$1$ &$s_1$ &$s_2$ &$s_1s_2$ &$(s_1s_2)^2$ &$(s_1s_2)^3$\\
\hline
$\phi_{(1,0)}$ &$1$ &$1$ &$1$ &$1$ &$1$ &$1$\\
\hline
$\phi_{(1,3)}'$ &$1$ &$-1$ &$1$ &$-1$ &$1$ &$-1$\\
\hline
$\phi_{(1,3)}''$ &$1$ &$1$ &$-1$ &$-1$ &$1$ &$-1$\\
\hline
$\phi_{(1,6)}$ &$1$ &$-1$ &$-1$ &$1$ &$1$ &$1$\\
\hline
$\phi_{(2,1)}$ &$2$ &$0$ &$0$ &$1$ &$-1$ &$-2$\\
\hline
$\phi_{(2,2)}$ &$2$ &$0$ &$0$ &$-1$ &$-1$ &$2$\\
\hline
\end{tabular}
\end{table}
Denote the Coxeter element $c(G_2)=s_1s_2$. The elliptic conjugacy classes are $c(G_2)$, $c(G_2)^2$, and $c(G_2)^3$ of sizes $2$, $2$, $1$, respectively. With respect to the \emph{usual character pairing}, on orthonormal basis of $\overline R(W(G_2))_\el$ is given by 
\[\{\sqrt 6\cdot 1_{c(G_2)},\sqrt 6\cdot 1_{c(G_2)^2},\sqrt{12}\cdot 1_{c(G_2)^3}\},\]
where
\begin{equation}
\begin{aligned}
1_{c(G_2)}&=\frac 16(\phi_{(1,0)}-\phi_{(1,3)}'-\phi_{(1,3)}''+\phi_{(1,6)}+\phi_{(2,1)}-\phi_{(2,2)}),\\
1_{c(G_2)^2}&=\frac 16(\phi_{(1,0)}+\phi_{(1,3)}'+\phi_{(1,3)}''+\phi_{(1,6)}-\phi_{(2,1)}-\phi_{(2,2)}),\\
1_{c(G_2)^3}&=\frac 1{12}(\phi_{(1,0)}-\phi_{(1,3)}'-\phi_{(1,3)}''+\phi_{(1,6)}-2\phi_{(2,1)}+2\phi_{(2,2)}).\\
\end{aligned}
\end{equation}

There are three families of Weyl group representations:
\[\C F:\quad \{\phi_{(1,0)}\},\quad \{\phi_{(2,1)},\phi_{(1,3)}',\phi_{(1,3)}'',\phi_{(2,2)}\},\quad \{\phi_{(1,6)}\}.
\]
The finite groups $\Gamma_{\C F}$ associated to them are:
\[\Gamma_{\C F}: \quad 1,\quad S_3,\quad 1.
\]
We use the notation of \cite{orange} for representatives of the conjugacy classes in $S_3$ and the representations of the centralizers. Let $1$ be the identity in $S_3$, $g_2$ be a transposition and $g_3$ a $3$-cycle. The centralizers are $S_3$, $\bZ/2\bZ=\{1,g_2\}$ and $\bZ/3\bZ=\{1,g_3,g_3^2\}$, respectively. Let $1$ denote the trivial representation for each of these groups. Let $r$, $\epsilon$ denote the  standard, respectively sign representation  of $S_3$. Let $\epsilon$ denote the nontrivial character of $\bZ/2\bZ$ and $\theta$ the representation $g_3\mapsto \theta$ of $\bZ/3\bZ$, where $\theta$ is a nontrivial third root of unity.

Under the parameterization from Theorem \ref{t:finite-lusztig}, the unipotent minimal principal series representations attached to irreducible characters of $W(G_2)$ in the middle family correspond to pairs in $M(S_3)$ as follows:
\begin{equation}
\rho_{(2,1)}\to (1,1),\quad \rho_{(1,3)}'\to (1,r),\quad \rho_{(1,3)}''\to (g_3,1), \quad\rho_{(2,2)}\to (g_2,1).
\end{equation}
The remaining four unipotent representations are all cuspidal, and they correspond to pairs as follows (\cite[page 478]{Ca}):
\begin{equation}
G_2[1]\to (1,\epsilon), \quad G_2[-1]\to (g_2,\epsilon),\quad G_2[\theta]\to (g_3,\theta),\quad G_2[\theta^2]\to (g_3,\theta^2).
\end{equation}
The nonabelian Fourier transform for $S_3$ in the ordered set
\[(1,1),\ (1,r),\ (1,\epsilon),\ (g_2,1),\ (g_2,\epsilon),\ (g_3,1),\ (g_3,\theta),\ (g_3,\theta^2)\]
is:
\begin{equation}
\left(
\begin{matrix}
\frac 16 &\frac 13 &\frac 16 &\frac 12 &\frac 12 &\frac 13 &\frac 13 &\frac 13\\
\frac 13 &\frac 23 &\frac 13 &0 &0 &-\frac 13 &-\frac 13 &-\frac 13\\
\frac 16 &\frac 13 &\frac 16 &-\frac 12 &-\frac 12 &\frac 13 &\frac 13 &\frac 13\\
\frac 12 &0 &-\frac 12 &\frac 12 &-\frac 12 &0&0&0\\
\frac 12 &0 &-\frac 12 &-\frac 12 &\frac 12 &0&0&0\\
\frac 13 &-\frac 13 &\frac 13 &0&0 &\frac 23 &-\frac 13 &-\frac 13\\
\frac 13 &-\frac 13 &\frac 13 &0&0 &-\frac 13 &\frac 23 &-\frac 13\\
\frac 13 &-\frac 13 &\frac 13 &0&0 &-\frac 13 &-\frac 13 &\frac 23\\
\end{matrix}
\right)
\end{equation}
All $\Delta(x,\sigma)=1$ for $G_2$.

\subsection{Finite elliptic Fourier transform}\label{s:finite-elliptic} For every $F$-stable Levi subgroup $\C L\supset T_0$, let $i_{\C L^F}^{\C G^F}:R_u(\C L^F)\to R_u(\C G^F)$ and $r_{\C G^F}^{\C L^F}:R_u(\C G^F)\to R_u(\C L^F)$ be the parabolic induction, respectively parabolic restriction maps. (These are well defined since induction/restriction take unipotent representations to unipotent representations.) Define the space of elliptic unipotent representations 
\begin{equation}
\overline R_u(\C G^F)_\el=\bigcap_{\C L\neq \C G}\ker r_{\C G^F}^{\C L^F}.
\end{equation}
The space $\overline R_u(\C G^F)$ can be naturally identified with the space of unipotent class functions of $\C G^F$ which vanish on the nonelliptic conjugacy classes of $\C G^F.$ This notion of elliptic is compatible with the notion of elliptic representations of the finite Weyl group $\C W$, namely, the injection $R(\C W)\hookrightarrow R_u(\C G^F)$ induces an injection $\overline R(\C W)_\el\hookrightarrow \overline R_u(\C G^F)_\el$. This follows at once from the fact that the injections $R(\C W)\hookrightarrow R_u(\C G^F)$ are compatible with parabolic induction and restriction.

The following fact is extracted from the works of Lusztig \cite{char-sheaves} and Shoji \cite{Sh} by Moeglin and Waldspurger \cite[sections 2.7, 2.8, 3.16, 4.3]{MW}, see our discussion at the end of subsection \ref{s:finite-groups}.

\begin{proposition}
The nonabelian Fourier tranform preserves the elliptic space, i.e., $$\C F \C T (\overline R_u(\C G^F)_\el)=\overline R_u(\C G^F)_\el.$$
\end{proposition}

In light of this proposition, let $\C F\C T_\el$ denote the restriction of $\C F\C T$ to $\overline R_u(\C G^F)_\el.$ 

\begin{example}
We compute $\C F\C T_\el$ when $\C G=G_2$. A basis of $\overline R_u(G_2({\mathbb F}_q))_\el$, orthonormal with respect to the character pairing, is given by:
\begin{equation}
\C B_\el(G_2({\mathbb F}_q))=\{\sqrt 6\cdot \rho_{c(G_2)},\ \sqrt 6\cdot \rho_{c(G_2)^2},\ \sqrt{12}\cdot \rho_{c(G_2)^3},\ G_2[1],\ G_2[-1],\ G_2[\theta],\ G_2[\theta^2]\},
\end{equation}
where by $\rho_{c(G_2)^i}$ we denote the class function in $R_u(G_2^F)$ corresponding to $1_{c(G_2)^i}$. In other words, as a virtual representation, $\rho_{c(G_2)}=\frac 16(\rho_{(1,1)}-\rho_{(1,3)}'-\rho_{(1,3)}''+\rho_{(1,6)}+\rho_{(2,1)}-\rho_{(2,2)})$ etc. In this basis, the elliptic nonabelian Fourier transform is:
\begin{equation}
\C F\C T_\el(G_2({\mathbb F}_q))=
\left(\begin{matrix} \frac 16 &\frac 12 &\frac{\sqrt 2}3 &0 &\frac 1{\sqrt 6}&\frac 1{\sqrt 6}&\frac 1{\sqrt 6}\\
\frac 12 &\frac 12 &0 &\frac 1{\sqrt 6} &0 &-\frac 1{\sqrt 6} &-\frac 1{\sqrt 6}\\
\frac{\sqrt 2}3 &0 &\frac 13 &-\frac 1{\sqrt 3} &-\frac 1{\sqrt 3} &0 &0\\
0 &\frac 1{\sqrt 6} &-\frac 1{\sqrt 3} &\frac 16 &-\frac 12 &\frac 13 &\frac 13\\
\frac 1{\sqrt 6} &0 &-\frac 1{\sqrt 3} &-\frac 12 &\frac 12 &0&0\\
\frac 1{\sqrt 6} &-\frac 1{\sqrt 6} &0 &\frac 13 &0 &\frac23 &-\frac 13\\
\frac 1{\sqrt 6} &-\frac 1{\sqrt 6} &0 &\frac 13 &0 &-\frac 13 &\frac 23
\end{matrix}\right).
\end{equation}
This is an orthogonal, symmetric matrix.
\end{example}

\subsection{Simple $p$-adic groups} We return to the setting of $p$-adic groups. Since the classification of unipotent representations recalled below has only been obtained for simple (or quasisimple) groups, we will assume that $G$ is a simple split $p$-adic group. 

\begin{definition}
An irreducible smooth representation $(\pi,V)$ of $G$ is called \emph{unipotent} if there exists a parahoric subgroup $P_c$ of $G$ and a cuspidal unipotent representation $\rho$ of $M_c$ such that $\Hom_{M_c}(\rho,V^{U_c})\neq 0.$ 

Let $\Irr_u G$ denote the set of characters of irreducible unipotent $G$-representations and let $R_u(G)$ be their complex span.  Let $\overline R_u(G)_\el$ denote the image of $R_u(G)$ in $\overline R(G)_\el$. 
\end{definition}
By \cite[Theorem 6.11]{MP}, the functors of parabolic induction and the Jacquet functors map unipotent representations to unipotent representations.

\smallskip

Let $\Irr_{u,\temp} G$ denote the set of characters of irreducible tempered unipotent $G$-representations and define $R_u(G)_\temp$ to be their $\bC$-span. As before, we may identify $\overline R_u(G)_\el$ with the image of the Bernstein $A$-operator, and since every class in $\overline R(G)_\el$ is represented by a tempered character, we may think of the space of elliptic representations as
\[\overline R_u(G)_\el=A(R_u(G)_\temp)\subset R_u(G)_\temp.\]
Let $G^\vee$ denote the complex simple group dual to $G$. If $x\in G^\vee$, we will write $x=sn$ for the Jordan decomposition, $s$ semisimple, $n$ unipotent. Denote 
\[A_x=Z_{G^\vee}(x)/Z_{G^\vee}(x)^0 Z(G^\vee)=\C S_{G^\vee}(x)/Z(G^\vee),
\] 
where $\C S_{G^\vee}(x)$ is the group of components of $Z_{G^\vee}(x).$ 

We say that $s$ is compact if $s$ lies in a compact subgroup of $G^\vee.$ 
Denote by $\C T_{G^\vee}$ the set of elements of $G^\vee$ with compact semisimple part. Finally, we say that $x\in G^\vee$ is elliptic if $Z_{G^\vee}(x)$ does not contain any nontrivial torus. 
The parameterization of irreducible unipotent tempered representations is as follows.

\begin{theorem}[Langlands-Kazhdan-Lusztig classification, \cite{KL,Lu2}, \cite{Op,Op2}]\label{t:DLL} Suppose that $G$ is simple, split and adjoint. There exists a natural bijection between $\Irr_{u,\temp}G$ and $G^\vee$-orbits on 
\[\{(x,\phi)\mid x\in \C T_{G^\vee},\ \phi\in \widehat A_x\}.
\]
In this bijection, the irreducible square integrable representations correspond to the elliptic elements $x\in G^\vee$.
\end{theorem}
Decompose
\begin{equation}
R_u(G)_\temp=\bigoplus_{x\in  \C T_{G^\vee}/G^\vee} R_u(G)_\temp^x,\quad \overline R_u(G)_\el=\bigoplus_{x\in \C T_{G^\vee}/G^\vee} \overline R_u(G)_\el^x.
\end{equation}
In the Iwahori case, Reeder \cite[Main Theorem]{Re} described the elements $x$ such that $\overline R_u(G)_\el^x\neq 0$. This description can be extended to the setting of unipotent representations of a simple, adjoint group $G$ by combining:
\begin{enumerate}
\item[(a)] the results of Opdam-Solleveld \cite{OS2} on the elliptic theory of affine Hecke algebras with arbitrary positive parameters, more precisely, the expression of the Euler-Poincar\' e pairing between two tempered modules in terms of the analytic R-group \cite[Theorem 6.5]{OS2};
\item[(b)] the Hecke algebra isomorphisms between the unipotent representations and the categories of modules for affine Hecke algebras with unequal parameters, Lusztig \cite[``the arithmetic/geometric correspondence'']{Lu2}, see also Opdam \cite[Theorem 3.4]{Op2}. 
\item[(c)] the equality of the analytic R-group with the geometric R-group, as in sections 8 and 9 of \cite{Re}, particularly, \cite[(9.2.1), (9.2.2), (9.2.3)]{Re}. 
\end{enumerate}
To state the result, we need one more definition.

\begin{definition}
A semisimple element $s\in G^\vee$ is called \emph{isolated} if $Z_{G^\vee}(s)$ is semisimple. A unipotent element $n\in G^\vee$ is called \emph{quasidistinguished} if it is the unipotent Jordan factor of an elliptic element $x\in G^\vee$. An elliptic unipotent element is called \emph{distinguished}. Denote by $\C T^0_{G^\vee}\subset \C T_{G^\vee}$ the set of elements $x=sn\in G^\vee$ such that $s$ is isolated and $n$ is quasidistinguished in $Z_{G^\vee}(s)$.
\end{definition}
 Every elliptic element $x\in G^\vee$ has a Jordan decomposition $x=sn$ where $s$ is isolated and $n$ is distinguished in $Z_{G^\vee}(s)$, hence $\C T^0_{G^\vee}$ contains all the elliptic elements of $G^\vee$. Notice however that if $x=n$ is quasidistinguished, but not distinguished in $G^\vee$, then $x$ is not elliptic, but $x\in \C T^0_{G^\vee}$.
 
\begin{theorem}[{\cite[Main Theorem]{Re}}]\label{t:reeder}
The space $\overline R_u(G)_\el^x$ is nonzero if and only if $x\in \C T^0_{G^\vee}.$
\end{theorem}

Let $\C U^\el_{G^\vee}$ denote the set of conjugacy classes of elements $n$ such that $n$ is the unipotent Jordan part of an element of  $x\in \C T^0_{G^\vee}.$ For each $n\in \C U^\el_{G^\vee}$, denote
\begin{equation}\label{e:ell-decomp}
\overline R_u(G)_\el^n=\bigoplus_{x=sn\in  \C T^0_{G^\vee}}\overline R_u(G)_\el^x.
\end{equation}

\begin{example}
\begin{enumerate}
\item Suppose $G^\vee=Sp(2n,\bC)$. The set of conjugacy classes of unipotent elements in $G^\vee$ is in one-to-one correspondence, via the Jordan canonical form, with partitions of $2n$ such that the odd parts appear with even multiplicity. The distinguished unipotent classes correspond to partitions where all parts are even and distinct. The quasidistinguished unipotent classes correspond to partitions where all parts are even and the multiplicity of each part is at most $2$. On the other hand, the set $\C U^\el_{G^\vee}$ is in bijection with partitions of $2n$ consisting of only even parts and where the multiplicity of each part is at most $4$.
\item If $G^\vee=G_2$, then $\C U^\el_{G^\vee}=\{G_2,G_2(a_1)\}$, where $G_2$ denotes the regular unipotent class and $G_2(a_1)$ the subregular unipotent class. Both classes are distinguished. 
\end{enumerate}
\end{example}

Let $x=sn\in G^\vee$ be an elliptic element. Choose a Lie homomorphism $\psi: SL(2)\to Z_{G^\vee}(s)$ such that $\psi(\left(\begin{matrix}1&1\\0&1\end{matrix}\right))=n$. Let $M=Z_{G^\vee}(\psi)$.  Then $\C S_{G^\vee}(n)=M/M^0$ as it is well known. Moreover,  $M^0$ is a torus and $Z_M(s)$ is finite \cite[Lemma 7.1]{Re2}. In particular, $Z_{G^\vee}(s,\psi)^0=Z_M(s)^0=\{1\}$, hence $\C S_{G^\vee}(sn)=\C S_{G^\vee}(s,\psi)=Z_M(s)$. Again by  {\it loc. cit.}, we have a natural surjective map
\begin{equation}\label{S-groups}
\C S_{G^\vee}(sn)=Z_M(s)\longrightarrow Z_{\C S_{G^\vee}(n)}(s)
\end{equation}
whose kernel is $Z_{M^0}(s)$, a finite group. When $n$ is distinguished in $G^\vee$, equivalently $M^0=\{1\}$, then 
\begin{equation}\label{comp-groups}
\C S_{G^\vee}(sn)\cong Z_{\C S_{G^\vee}(n)}(s).
\end{equation}
\begin{lemma}\label{l:M-param}
Suppose $n$ is distinguished in $G^\vee$. Let $\Sigma_n$ be the set of $Z_{G^\vee}(n)$-orbits on $$\{(s,\phi):\ s\in
Z_{G^\vee}(n)\text{ semisimple},\ \phi\in\widehat{\C S_{G^\vee}(sn)}\}.$$ Then $\Sigma_n$ can be naturally identified with $M(\C S_{G^\vee}(n)).$
\end{lemma}
\begin{proof}
By (\ref{comp-groups}), $\C S_{G^\vee}(sn)\cong Z_{\C S_{G^\vee}(n)}(s).$ Recall that the elements of $M(\C S_{G^\vee}(n))$ are $\C S_{G^\vee}(n)$-orbits of pairs $(y,\tau)$, where $y$ is an element of $\C S_{G^\vee}(n)$ and $\tau$ is an irreducible representation of $Z_{\C S_{G^\vee}(n)}(y).$ The claim is immediate.
\end{proof}

\begin{example}\label{e:G2-params}
In the case $G=G_2$, the spaces $\overline R_u(G)_\el^n$, $n\in \{G_2,G_2(a_1)\}$ admit bases consisting of square integrable representations. The space $\overline R_u(G)_\el^{G_2}$ is one-dimensional, spanned by $v_1$, the Steinberg representation. The space $\overline R_u(G)_\el^{G_2(a_1)}$ is $8$-dimensional. The eight square integrable are as follows: $4$ have Iwahori-fixed vectors and $4$ are supercuspidal. The $4$ supercuspidal representations are compactly induced $\text{c-ind}_{K_0}^G(\wti\rho)$, where $K_0$ is the maximal hyperspecial compact open subgroup $K_0=G(\CO)$ and $\wti\rho$ is the pull-back to $K_0$ one of the $4$ cuspidal unipotent representations of $G_2({\mathbb F}_q)$. 

The group of components is $\C S_{G^\vee}(n)=S_3.$ The set $M(S_3)$ has precisely cardinality $8$, which is the same as the number of  irreducible unipotent square integrable representations with unipotent part $n$. This is not a coincidence, since $\Sigma_n\cong M(\C S_{G^\vee}(n))$ by Lemma \ref{l:M-param}.

\smallskip

There are $3$ isolated semisimple conjugacy classes in $G_2$, whose representatives we denote by
$s_0=1$, $s_1$, and $s_2.$ The corresponding centralizers in
$G^\vee=G_2$ have types $G_2$, $A_1\times \wti A_1,$ and $A_2$,
respectively.
The parameterizations and the division of the $8$ unipotent discrete series representations into L-packets are in Table \ref{t:G2-packets}. We label the representations $v_2,\dots,v_9$ for later reference.

\begin{table}[h]
\caption{Unipotent parameters for $n=G_2(a_1)$, $A_n=S_3$, in $G_2$\label{t:G2-packets}}
\begin{tabular}{|c|c|c|c|c|}
\hline
$Z_{G^\vee}(s)$ &$A_{su}$ &$ \phi\in\widehat{A_{su}}$ &$(y,\tau)\in
M(A_u)$ &Label; Representation \\
\hline
\hline
$G_2$ &$S_3$ &$1$ &$(1,1)$ &$v_2$; Iwahori, generic, dual of the affine reflection repn.\\
\hline
&&$\refl$ &$(1,r)$ &$v_3$; Iwahori, nongeneric, long reflection sign repn.\\
\hline
&&$\sgn$ &$(1,\epsilon)$ &$v_6$; supercuspidal $G_2[1]$\\
\hline
\hline
$A_1+\wti A_1$ &$\bZ/2\bZ$ &$1$ &$(g_2,1)$ &$v_5$; Iwahori, endoscopic
$A_1\times\wti A_1$\\
\hline
&&$\sgn$ &$(g_2,\epsilon)$ &$v_7$; supercuspidal $G_2[-1]$\\
\hline 
\hline
$A_2$ &$\bZ/3\bZ$ &$1$ &$(g_3,1)$ &$v_4$; Iwahori, endoscopic $A_2$\\
\hline
&&$\zeta$ &$(g_3,\theta)$ &$v_8$; supercuspidal $G_2[\theta]$\\
\hline
&&$\zeta^2$ &$(g_3,\theta^2)$ &$v_9$; supercuspidal $G_2[\theta^2]$\\

\hline
\end{tabular}
\end{table}

\end{example}

\subsection{The elliptic restriction map} Suppose $(\pi,V)$ is a unipotent representation of $G$. For every standard parahoric subgroup $P_J$, $J\subsetneq \Pi_a$, we define the \emph{restriction} of $\pi$ to be:
\begin{equation}
\res_{P_J}(\pi)=\text{the character of the }M_J\text{-representation on }V^{U_J}.
\end{equation}
We recall an important result of Moy and Prasad.
\begin{lemma}[{see \cite[Theorem 3.5]{MP}}]\label{l:MP} Let $G$ be a simple $p$-adic group and $(\pi,V)$ an admissible $G$-representation. Suppose that $J\subsetneq \Pi_a$ and that $V^{U_J}$ contains the cuspidal unipotent representation $\mu$ of $M_J$. If $J'\subsetneq \Pi_a$ is such that $V^{U_{J'}}\neq 0$ and $J'$ is minimal with respect to this property, then $J'=\omega J \omega^{-1}$ for some $\omega\in \Omega$ and $V^{U_{J'}}$ is a direct sum of copies of the twist $\mu^\omega$.  
\end{lemma}
This means that the image of the restriction lands in the unipotent $M_J$-characters, i.e., it defines a map $\res_{P_J}: R_u(G)\to R_u(M_J).$ Let 
\begin{equation}
\proj_\el^{M_J}: R_u(M_J)\to \overline R_u(M_J)_\el
\end{equation}
be the projection with respect to the ordinary character pairing. This is the same as defining $\proj_\el^{M_J}(\sigma)(\C C)=\tr\sigma(\C C)$ if $\C C\subset M_J$ is an elliptic conjugacy class and $\proj_\el^{M_J}(\sigma)(\C C)=0$ otherwise.
For the computations later on, if $\C B_\el$ is an orthonormal (with respect to the character pairing) basis of $\overline R_u(M_J)_\el$ (viewed as a subspace of $R_u(M_J)$), then $\proj_\el(\sigma)=\sum_{\rho\in \C B_\el}\langle\sigma,\rho\rangle_{M_J}\cdot \rho$.

\begin{definition} Suppose that $G$ is simple and simply-connected. 
Let $\Pi_{a,\max}$ denote the set of maximal subsets $J\subsetneq \Pi_a$. Define the \emph{unipotent elliptic restriction map}:
\begin{equation}
\begin{aligned}
\res_{u,\el}:~ &\overline R_u(G)_\el=A(R_u(G)_\temp)\longrightarrow \bigoplus_{J\in \Pi_{a,\max}}\overline R_u(M_J)_\el,\\ \pi&\longmapsto \sum_{J\in \Pi_{a,\max}} (\proj_\el^{M_J}\circ \res_{P_J})(\pi). 
\end{aligned}
\end{equation}
\end{definition}

Define $\overline \CH(G)^\el_u$ to be the image of $\overline R(G)_u^\el$ under the homomorphism $\tau_{HS}\circ\ep$. In other words, $\overline \CH(G)^\el_u$ is the span in $\overline \CH(G)$ of all the pseudocoefficients $f_\pi$ of elliptic unipotent representations $\pi$ of $G$. 

On the other hand, clearly extension by zero outside $P_J$ realizes $\overline R_u(M_J)_\el$ as a subspace of $\overline\CH_c(G).$ Let  $\iota:\bigoplus_{J\in \Pi_{a,\max}}\overline R_u(M_J)_\el\to \overline\CH_c(G)$ be the linear map defined by $\sum_J f_J\mapsto \wti f_J.$ 
The results in the following proposition are proved in \cite[Theorem 1.9]{MW}.

\begin{proposition}[{\cite{MW}}]\label{p:pseudo}
Suppose $G$ is simple and simply-connected. Then:
\begin{enumerate}
\item $\iota$ is injective and the image is $\overline\CH(G)^\el_u.$ In particular, we may identify $\overline\CH(G)^\el_u=\bigoplus_{J\in \Pi_{a,\max}}\overline R_u(M_J)_\el$.
\item $f_\pi\equiv \res_{u,\el}(\pi)$ in $\overline\CH(G)^\el_u$ for every $\pi\in \overline R_u(G)_\el$;
\item The map $\res_{u,\el}: \overline R_u(G)_\el\to \overline \CH(G)^\el_u$ is an isometry with respect to the Euler-Poincar\' e pairing $\EP$ for $\overline R_u(G)_\el$ and the ordinary character pairing for $\overline\CH(G)^\el_u=\bigoplus_{J\in \Pi_{a,\max}}\overline R_u(M_J)_\el$.
 \end{enumerate}
\end{proposition}

An analogous result was proved in \cite[Section 3]{OS} in the setting of the affine Hecke algebra $\CH$  of an affine Weyl group $W$ and with arbitrary positive parameters. In that case, one has an isometric isomorphism $\overline R(\CH)_\el\to \overline R(W)_\el$ with respect to the Euler-Poincar\' e pairings on both spaces. Moreover, $\overline R(W)_\el$ is naturally isometrically isomorphic with the direct sum 
\begin{equation}\label{e:affine-Weyl}
\overline R(W)_\el\cong\bigoplus_v\overline R(W_v)_\el
\end{equation}
over the vertices $v$ of the fundamental alcove $c_0$. Here $W_v$ is the finite Weyl group which centralizes in $W$ the vertex $v$, but the pairing for $\overline R(W_v)_\el$ needed for the isometry in (\ref{e:affine-Weyl}) is the ordinary character pairing of the finite Weyl group.

\subsection{The elliptic nonabelian Fourier transform for $G_2$} In this subsection, we investigate the behavior of the nonabelian Fourier transform with respect to the restriction map just defined.

\begin{definition}
In light of Proposition \ref{p:pseudo}, define the \emph{elliptic unipotent nonabelian Fourier transform} of $G$ to be
\begin{equation}
\C F\C T_{u,\el}: \overline\CH(G)^\el_u\longrightarrow \overline\CH(G)^\el_u,\quad \C F\C T_{u,\el}=\bigoplus_{J\in \Pi_{a,\max}} \left(\C F\C T_\el: \overline R_u(M_J)_\el\to \overline R_u(M_J)_\el\right).
\end{equation}
\end{definition}

Now suppose that $n\in \C U_{G^\vee}^\el$ is distinguished and special. As explained in Example \ref{e:G2-params}, we may identify the parameterizing set of $\Irr_u(G)_\el^n$ with $M(A_n)$. For a square integrable representation $\pi\in \Irr_u(G)_\el^n$, let $\lambda_\pi\in M(A_n)$ denote the resulting parameter.  We can define a {\it dual Fourier transform} on $\overline R_u(G)_\el^n$ via:
\begin{equation}
\C F\C T^{\vee,n}_{u,\el}: \overline R_u(G)_\el^n\to \overline R_u(G)_\el^n,\quad \pi\mapsto \sum_{\pi'\in \Irr_u(G)_\el^n} \{\lambda_\pi,\lambda_{\pi'}\}_{A_n}\cdot \pi',
\end{equation}
where $\{~,~\}_{A_n}$ denotes the nonabelian Fourier transform for the group $A_n$.

\begin{definition}
Suppose $G=G_2$. By Theorem \ref{t:reeder} and (\ref{e:ell-decomp}), $\overline R_u(G)_\el=\bigoplus_{n\in \C U_{G^\vee}^\el} \overline R_u(G)_\el^n$. By Example \ref{e:G2-params}, we may consider Lusztig's nonabelian Fourier transform, which we denote $\C F\C T^{\vee,n}_{u,\el}$, with respect to the parameter space  $\Sigma_n=M(\C S_{G^\vee}(n))$ for $n\in \C U_{G^\vee}^\el$. Define the \emph{dual elliptic unipotent nonabelian Fourier transform} to be
\begin{equation}
\C F\C T^{\vee}_{u,\el}: \overline R_u(G)_\el\to \overline R_u(G)_\el,\quad \C F\C T^{\vee}_{u,\el}=\bigoplus_{n\in \C U_{G^\vee}^\el} \C F\C T^{\vee,n}_{u,\el}.
\end{equation}
\end{definition}

We remark that, in the natural bases described above, the map $\C F\C T^{\vee}_{u,\el}$ is block diagonal $$\C F\C T^{\vee}_{u,\el}=\left(\begin{matrix}\{~,~\}_1 &0\\ 0&\{~,~\}_{S_3}\end{matrix}\right)$$ with blocks of sizes $1$ and $8$, while $\C F\C T_{u,\el}$ is block diagonal $$\C F\C T_{u,\el}=\left(\begin{matrix}\C F\C T_\el(G_2({\mathbb F}_q)) &0 &0\\0&\C F \C T_\el((A_1+\wti A_1)({\mathbb F}_q)) &0\\ 0&0&\C F\C T_\el(A_2({\mathbb F}_q))\end{matrix}\right)$$ with blocks of sizes $7$, $1$, $1$.

\begin{theorem} Suppose $G=G_2$. The diagram
\begin{equation}
\xymatrixcolsep{5pc}\xymatrix{
\overline R_u(G)_\el \ar[d]_{\res_{u,\el}} \ar[r]^{\C F\C T^{\vee}_{u,\el}} &\overline R_u(G)_\el \ar[d]^{\res_{u,\el}}\\
\overline\CH(G)^\el_u \ar[r]_{\C F\C T_{u,\el}} &\overline\CH(G)^\el_u 
}
\end{equation}
is commutative.
\end{theorem}

\begin{proof} The proof is based on a direct computation. The restrictions $\res_{P_J}(v_i)$ are given in Table \ref{t:G2-res}. This has been computed using the  result of Moy-Prasad, see Lemma \ref{l:MP}, and the reduction to Iwahori-Hecke algebras. 
In particular, if we apply Lemma \ref{l:MP} to the supercuspidal representations $v_i=\text{c-ind}_{P_{J_0}}^{G_2}(\mu_i)$, $6\le i\le 9$, where $\mu_i$ is a cuspidal unipotent representation of $G_2({\mathbb F}_q)$ as in the table, then we see that $\res_{P_{J_j}}(v_i)=\delta_{i,j}\mu_i$.

The representations $v_1-v_5$ are all Iwahori-spherical. Let $\CH(G,I)$ denote the Iwahori-Hecke algebra of compactly supported, smooth, $I$-biinvariant functions \cite{Bo,IM}. This is isomorphic to the affine Hecke algebra of type $G_2$ with equal parameters. Let $\CH(P_J,I)$ denote the subalgebra of $\CH(G,I)$ of functions whose support is in $P_J$. This is isomorphic to the finite Hecke subalgebra of type $W_J$. The subspace of Iwahori-fixed vectors $v_i^I$ is naturally an $\CH(G,I)$-module and therefore, we may consider the restriction of $v_i^I$ to $\CH(P_J,I)$.

The restriction $\res_{P_{J_j}}(v_i)$ can be computed at the level of the Iwahori-Hecke algebra. More precisely, suppose that $\mu$ is an $M_J$-type such that $\mu$ appears in $\res_{P_{J}}(v_i)$, $1\le i\le 5.$ Let $\wti\mu$ denote the pullback of $\mu$ to $P_J$. Then, by Lemma \ref{l:MP}, $\wti\mu^I\neq 0$. Moreover:
\begin{equation}
\dim\Hom_{M_J}(\mu,v_i^{U_{J}})=\dim\Hom_{\CH(P_J,I)}(\wti\mu^I, v_i^I).
\end{equation}
The structure of $v_i^I$ as Hecke algebra modules is well known, and this is how we compute these restrictions.
\begin{table}[h]
\caption{Restrictions of unipotent elliptic $G_2$-representations\label{t:G2-res}}
\begin{tabular}{|c|c|c|c|}
\hline
$\Irr_u(G_2)_\el$ &$J_0=\{\al_1,\al_2\}=G_2$ &$J_1=\{\al_0,\al_2\}=A_1+\wti A_1$ &$J_2=\{\al_0,\al_1\}=A_2$\\
\hline
$v_1$ &$\phi_{(1,6)}$ &$\sgn_0\otimes \sgn_2$ &$\sgn$\\
\hline
$v_2$ &$\phi_{(1,6)}+\phi_{(2,1)}$ &$\sgn_0\otimes \sgn_2+\sgn_0\otimes\triv_2+\triv_0\otimes \sgn_2$ &$\sgn+\refl$  \\
\hline
$v_3$ &$\phi_{(1,3)}'$ &$\sgn_0\otimes \triv_2$ &$\sgn$\\
\hline
$v_4$ &$\phi_{(1,6)}+\phi_{(1,3)}''$ &$\triv_0\otimes\sgn_2+\sgn_0\otimes\sgn_2$ &$\refl$ \\
\hline
$v_5$ &$\phi_{(1,6)}+\phi_{(2,2)}$ &$\sgn_0\otimes \sgn_2+\sgn_0\otimes\triv_2+\triv_0\otimes\sgn_2$ &$\sgn+\refl$\\
\hline
$v_6$ &$G_2[1]$ &$0$ &$0$\\
\hline
$v_7$ &$G_2[-1]$ &$0$ &$0$\\
\hline
$v_8$ &$G_2[\theta]$ &$0$ &$0$\\
\hline
$v_9$ &$G_2[\theta^2]$ &$0$ &$0$\\
\hline
\end{tabular}
\end{table}

Next, once we have $\res_{P_J}(v_i)$, we need to compute $\res_{u,\el}(v_i)$ in terms of the orthonormal bases $\C B_\el(G_2({\mathbb F}_q))$ and 
\begin{equation}
\C B_\el((A_1+\wti A_1)({\mathbb F}_q))=\left\{\frac 12 (\triv_0-\sgn_0)\otimes (\triv_2-\sgn_2)\right\} \text{ and } \C B_\el(A_2({\mathbb F}_q))=\left\{\frac 1{\sqrt 3} (\triv-\refl+\sgn)\right\}.
\end{equation}
This is done by projecting $\sum_{J\text{ max}}\res_{P_J}(v_i)$ onto the elliptic space. In terms of the ordered bases listed before, the matrix of $\res_{u,\el}$ is the orthogonal matrix:
\begin{equation}
[\res_{u,\el}]=\left(\begin{matrix}
1/\sqrt{6} & 2/\sqrt{6} &-1/\sqrt{6} &0 &0 &0 &0 &0 &0\\
1/\sqrt{6} &0 &1/\sqrt{6} &2/\sqrt{6} &0 &0 &0 &0 &0\\
1/2\sqrt{3} &-1/2\sqrt{3} &-1/2\sqrt{3} &0 &\sqrt{3}/2 &0 &0 &0 &0\\
 0 &0 &0 &0 &0 &1 &0 &0 &0\\
 0 &0 &0 &0 &0 &0 &1 &0 &0\\
  0 &0 &0 &0 &0 &0 &0 &1 &0\\
  0 &0 &0 &0 &0 &0 &0 &0 &1\\
  1/2 &-1/2 &-1/2 &0 &-1/2 &0 &0 &0 &0\\
  1/\sqrt{3} &0 &1/\sqrt{3} &-1/\sqrt{3} &0 &0 &0 &0 &0
  \end{matrix}
  \right)
\end{equation}
The claim now amounts to the direct verification that 
\[[\res_{u,\el}]\cdot [\C F\C T_{u,\el}]\cdot [\res_{u,\el}]^T=[\C F\C T^{\vee}_{u,\el}].
\]
\end{proof}

\ifx\undefined\bysame
\newcommand{\bysame}{\leavevmode\hbox to3em{\hrulefill}\,}
\fi

\end{document}